\documentclass[reqno]{amsart}
\usepackage{amsmath,amsthm,amsfonts,amssymb,srcltx}
\usepackage{graphicx} 
\usepackage{enumerate}
\usepackage{relsize}

\newtheorem{theorem}{Theorem}[section]
\newtheorem{lemma}[theorem]{Lemma}
\newtheorem{prop}[theorem]{Proposition}
\newtheorem{cor}[theorem]{Corollary}

\theoremstyle{definition}
\newtheorem{definition}[theorem]{Definition}
\newtheorem{example}[theorem]{Example}

\theoremstyle{remark}
\newtheorem{remark}[theorem]{Remark}

\newcommand{\Exp}{\mathbb{E}\,}
\newcommand{\sign}{sgn}
\newcommand{\Tr}{\operatorname{Tr}\,}
\newcommand{\abs}{\operatorname{abs}}
\newcommand{\ind}[1]{\mathlarger{\mathbb{I}}\{#1\}}
\renewcommand{\Im}{\operatorname{Im}}

\begin{document}

\title[Singular Values  of Squares of Elliptic Matrices]{Singular Values Distribution of Squares of Elliptic Random Matrices and Type B Narayana Polynomials}

\author{Nikita Alexeev}
\address{St. Petersburg State University, Russia \and George Washington University, USA}
\email{nikita.v.alexeev{\char'100}gmail.com}

\author{Alexander Tikhomirov}
\address{Komi Science Center of Ural Division of RAS, Russia}
\maketitle

\begin{abstract}
We consider Gaussian elliptic random matrices $X$ of a size $N \times N$ with parameter $\rho$, i.e., matrices whose pairs of entries $(X_{ij}, X_{ji})$ are mutually independent Gaussian vectors, 
$\Exp X_{ij} = 0$, $\Exp X^2_{ij} = 1$ and $\Exp X_{ij} X_{ji} = \rho$. We are interested in the asymptotic distribution of eigenvalues of the matrix $W =\frac{1}{N^2} X^2 X^{*2}$.
We show that this distribution is defined by its moments and we provide  a recurrent relation for these moments.
We prove  that the (symmetrized) asymptotic distribution is determined by its free cumulants, which are Narayana polynomials of type B:
$$c_{2n} = \sum_{k=0}^n {\binom{n}{k}}^2 \rho^{2k}.$$
\end{abstract}

\section{Introduction}
Consider a family of mutually independent real Gaussian vectors  $(X_{ij},X_{ji})$, $1 \leq i \leq j$, such that
\begin{enumerate}
 \item $\Exp X_{ij} = \Exp X_{ji} = 0$ ,
 \item $\Exp X^2_{ij} = \Exp X^2_{ji} = 1$ ,
 \item $\Exp X_{ij} X_{ji} = \rho$ , for any $i \neq j$ and some $-1 < \rho < 1$.
\end{enumerate}
Consider a square $N \times N$ random matrix $X$ with entries $X_{ij}$, for $1 \leq i,j \leq N$. 
This matrix ensemble was introduced by Girko in \cite{girko1985}  and was recently investigated (see, for example, \cite{girko2006strong,naumov2012,goetze2013,nguyen2014}). Such matrices are called  elliptic random   matrices due to the fact that their asymptotic spectral distribution is the uniform distribution inside an ellipse in the complex plane.   

Here we are interested in the singular values distribution of the random matrix $\frac{1}{N}X^2$. Namely, we denote by $\lambda_1 \leq \lambda_2 \leq \dots \leq \lambda_N$ the eigenvalues of the matrix $$W =\frac{1}{N^2} X^2 X^{*2} \; , $$
and define its empirical spectral distribution function by
$$\mathcal{F}_N(x) = \frac{1}{N} \sum_{i = 1} ^N \ind{\lambda_i \leq x} \; , \textrm{ where } \mathlarger{\mathbb{I}} \textrm{ stands for indicator function, }$$
 and we define the expected spectral distribution function by
$$F_N(x) = \Exp \mathcal{F}_N(x) \; . $$
 We are interested in the asymptotic distribution of squared singular values of $\frac{1}{N}X^2$, namely  $F(x) = \lim_{N \to \infty} F_N(x)$.
 Note that the distributions $\mathcal{F}_N, F_n$ and $F$  depend on $\rho$, but this dependency is not reflected in our notations.
 \begin{theorem}
  Let $X$ be an $N \times N$ real Gaussian elliptic random matrix and $F_N(x)$ be its expected spectral distribution function. Then the $\lim_{N \to \infty} F_N(x) = F(x)$ exists, and this distribution function $F(x)$ is uniquely determined by its moments 
  $$M_k(\rho) = \int_{-\infty}^{\infty} x^k d F(x) \; .$$
  The sequence of moments  $M_k(\rho)  = U_{2k}(\rho)$,  where polynomials $U_k(\rho)$ satisfy the recurrent relation:
  \begin{align}
   U_{k+1}(\rho) = \sum_{i =0}^{\left\lfloor\frac{k-1}{2}\right\rfloor} U_{k-2i-1}(\rho) V_{2i+1}(\rho) + \rho \sum_{i=0}^{\left\lfloor\frac{k}{2}\right\rfloor} V_{2i}(\rho) U_{k-2i}(\rho) \, , \nonumber \\  
   V_{k+1}(\rho) = \sum_{i=0}^{\left\lfloor\frac{k}{2}\right\rfloor} U_{2i}(\rho) V_{k-2i}(\rho)+\rho \sum_{i=0}^{\left\lfloor\frac{k-1}{2}\right\rfloor} U_{2i+1}(\rho)V_{k-2i-1}(\rho) \, ,
   \label{eq:rec}
  \end{align}

  with initial conditions $U_0(\rho) = V_0(\rho) =1.$
  \label{th:mom}
 \end{theorem}

 The first several moments $M_k(\rho)$ are:
 \begin{align*}
  M_1(\rho) &= 1+\rho^2 \\
  M_2(\rho) &= 3+8\rho^2+3\rho^4 \\
  M_3(\rho) &= 12 + 54 \rho^2 + 54 \rho^4 + 12 \rho^6 \\
  M_4(\rho) &= 55 + 352\rho^2 + 616 \rho^4 +352 \rho^6 + 55 \rho^8
 \end{align*}

 The case $\rho = 0$ is well known, and the corresponding distribution is the so-called Fuss--Catalan distribution  \cite{alexeev2010a,alexeev2010b,liu2011,mlotkowski2012}. Its moments are Fuss--Catalan numbers $\binom{3k}{k}\frac{1}{2k+1}$,  the sequence A001764 in OEIS \cite{oeis}. The case $\rho = 1$ is the case of Hermitian matrices $X = X^*$, and the moments of the corresponding distribution are Catalan numbers with even indices  $\binom{4k}{2k}\frac{1}{2k+1}$. Our case generalizes both of these cases. Note, that
 $$\binom{3k}{k}\frac{1}{2k+1} \leq M_k(\rho) \leq \binom{4k}{2k}\frac{1}{2k+1} \, .$$
 
 Let us consider the symmetrization of the distribution $F(x)$. Namely, let the random variable $\xi \geq 0$ have distribution $F(x)$. Let $\eta = \epsilon \sqrt{\xi}$, where 
 $\epsilon$ is a Rademacher random variable $\Pr\{\epsilon = \pm 1\} = \frac{1}{2}$  independent of $\xi$. The distribution function of $\eta$ is
 $$G(x) = \frac{1}{2} \left(1+\sign(x)F(x^2)\right) \, ,$$
 and its $2k^{th}$ moments are $M_k(\rho)$ and its $(2k+1)^{st}$ moments are $0$.
 
 Let us define free cumulants of any probability distribution $\mu$ with finite support. We start with the Cauchy transform $s_{\mu}(z)$ which is defined by the series
$$
s_{\mu}(z)=\sum_{m=0}^\infty \frac{M_m}{z^{m+1}} \, ,
$$
where $M_m$ are the moments of the measure $\mu$.
Then the $\mathcal{R}$\emph{-transform}, introduced by Voiculescu (see, e.g., \cite{voiculescu1992}), can be defined as the solution of the equation
$$\mathcal{R}_\mu (s_\mu(z))+\frac{1}{s_\mu(z)}=z \, .$$
Then  the free cumulants $\{c_n\}_{n=1}^\infty$  of the measure $\mu$ are given as the coefficients of a power series expansion of $\mathcal{R}_\mu (z)$:
$$\mathcal{R}_\mu (z) = \sum_{n =0}^\infty c_{n+1}z^n \, .$$

\begin{theorem}
 The distribution $G(x)$ is uniquely determined by its free cumulants $c_n(\rho)$. Odd free cumulants are equal to zero, and even free cumulants are Narayana polynomials of type B:
 $$c_{2n}(\rho) = \sum_{k=0}^{n} \binom{n}{k}^2 \rho^{2k} \, . $$
 The $\mathcal{R}$-transform of the distribution $G(x)$ is
 $$\mathcal{R}_G (z) = \frac{1}{z}\left(\frac{1}{\sqrt{\left((\rho^2-1) z^2-1\right)^2-4 z^2}}-1\right) \, .$$
 \label{th:cumul}
\end{theorem}

It is known that Narayana polynomials of type A
$$N^{A}_n(t) = \sum_{k=1}^n \frac{1}{k}\binom{n-1}{k-1} \binom{n}{k-1} t^k$$
appear as free cumulants of free Bessel law $\pi_{2,t}$ \cite{banica2011} and as moments of Marchenko--Pastur distribution \cite{pastur1967}. For combinatorial aspects of Narayana polynomials of types A and B see \cite{reiner1997}.


 \section{Moments}
 We use the method  of moments to prove Theorem \ref{th:mom}.
 The $k^{th}$ moment of the distribution $F_N(x)$ is
 $$M^{(N)}_k(\rho) = \int x^k dF_N(x) = \frac{1}{N}\Exp(\lambda_1^k+\lambda_2^k+\dots+\lambda_N^k) = \frac{1}{N}\Exp \Tr W^k \, .$$
 The trace of $W^k$ can  be written as
 \begin{align}
   \Tr W^k  =  \frac{1}{N^{2k}}\sum_{(4k)} {\prod_{j=0}^{k-1} X_{i_{4j} i_{4j+1}} X_{i_{4j+1} i_{4j+2}} X_{i_{4j+3} i_{4j+2}} X_{i_{4j+4} i_{4j+3}}} \, ,
   \label{eq:trace}
 \end{align}
 where the sum $\sum_{(4k)}$ is taken over all indices $\{i_0, i_1, \dots, i_{4k}\}$, such that $i_j \in \{1,2,\dots,N\}$ and $i_0 = i_{4k}$ . To compute the expectation of the trace, we will use Wick's formula (see \cite{zvonkin1997} for references).
 \begin{prop}[Wick's Theorem]
 Let $(x_1,x_2, \dots, x_{2n})$ be a zero-mean multivariate  Gaussian vector. Then the expectation
 $$\Exp (x_1x_2\cdots x_{2n-1}) = 0 , $$
 and
 $$\Exp (x_1x_2\cdots x_{2n}) = \sum_{\pi \in \mathcal{P}_{2n}} \prod_{(i,j) \in \pi} \Exp x_i x_j\, , $$
 where the sum $\sum_{\pi \in \mathcal{P}_{2n}}$ is taken over a set $\mathcal{P}_{2n}$ of all partitions of $\{1,2,\dots,2n\}$ into pairs. 
 \end{prop}

 Since the vector $(X_{i_0i_1}, X_{i_1i_2}, \dots, X_{i_0 i_{4k-1}})$ is a zero-mean multivariate Gaussian vector, one can apply Wick's formula to compute $\Exp\Tr W^k$. To do this we will represent a pair partition as a chord diagram on $4k$ vertices in the following way. The $j^{th}$ vertex ($1 \leq j \leq 4k$) of the chord diagram corresponds to the $j^{th}$  factor in the product in \eqref{eq:trace}. The $j^{th}$ vertex is colored white if $j \pmod{4} \in \{0,3\}$, i.e., if the corresponding factor is taken from the transposed matrix, and it is colored black if $j \pmod{4} \in \{1,2\}$. Two vertices are connected by a chord if the corresponding factors are paired. We use $p \overset{d}\sim q$ to denote that the $p^{th}$ and the $q^{th}$ vertices are connected in a diagram $d$. We call a chord diagram {\it planar} if its chords do not cross.
 We denote by $\mathcal{D}_{2k}$ the set of all such chord diagrams on $4k$ vertices and by $\mathcal{D}^0_{2k}$ the set of planar chord diagrams on $4k$ vertices. We say that the weight $\omega(d)$ of a planar chord diagram $d \in \mathcal{D}^0_{2k}$ is equal to $\rho^l$ if it has exactly $l$ chords connecting a pair of vertices of the same color.

 \begin{lemma}
  \begin{align}
   \lim_{N \to \infty}\frac{1}{N} \Exp Tr W^k = \sum_{d \in \mathcal{D}^0_{2k}} \omega(d) \, .
   \label{eq:weights}
  \end{align}

  \label{lem:comb}
 \end{lemma}
 \begin{proof}
 
 For the sake of simplicity we will use the denotation  $X_{\mathbf{j}}$
\begin{align*}
  X_{\mathbf{j}} = \begin{cases}
                   X_{i_{j-1} i_{j}} \textrm{, if } j \pmod{4} \in \{1,2\} \, , \\
                   X_{i_{j} i_{j-1}} \textrm{, if } j \pmod{4} \in \{0,3\} \, .
                  \end{cases}
\end{align*}

  \begin{align*}
  &\frac{1}{N}\Exp \Tr W^k  =  \frac{1}{N^{2k+1}}\sum_{(4k)} \Exp{\prod_{j=0}^{k-1} X_{i_{4j} i_{4j+1}} X_{i_{4j+1} i_{4j+2}} X_{i_{4j+3} i_{4j+2}} X_{i_{4j+4} i_{4j+3}}} \\
  & = \frac{1}{N^{2k+1}}\sum_{(4k)} \, \sum_{d \in \mathcal{D}_{2k}} \, \prod_{p \overset{d}\sim q} \Exp X_{\mathbf{p}}X_{\mathbf{q}} \\
  & = \sum_{d \in \mathcal{D}^0_{2k}} \frac{1}{N^{2k+1}}\sum_{(4k)} \, \prod_{p \overset{d}\sim q} \Exp X_{\mathbf{p}}X_{\mathbf{q}} \, +  \sum_{d \in \mathcal{D}_{2k}\setminus \mathcal{D}^0_{2k}} \frac{1}{N^{2k+1}}\sum_{(4k)} \, \prod_{p \overset{d}\sim q} \Exp X_{\mathbf{p}}X_{\mathbf{q}} \, .
  \end{align*}
The contribution of non-planar chord diagrams $d \in \mathcal{D}_{2k}\setminus \mathcal{D}^0_{2k}$ to the limit of $\frac{1}{N} \Exp Tr W^k$ is $0$. Indeed, all factors $\Exp X_{\mathbf{p}}X_{\mathbf{q}}$ are bounded; the number of summands in the sum  $\sum_{(4k)}$ is  equal to $N^f$, where $f$ is the number of independent indices $i_j$, or, in terms of chord diagrams, the number of boundary components. Since there is $N^{2k+1}$ in the denominator, a chord diagram makes an asymptotic nonzero contribution only if $f \geq 2k+1$. But
$f = 2k+1$ only for planar chord diagrams and $f$ is less for non-planar diagrams.

  For a planar chord diagram $d  \in \mathcal{D}^0_{2k}$, the number of independent indices $f = 2k+1$ and the product $\prod_{p \overset{d}\sim q} \Exp X_{\mathbf{p}}X_{\mathbf{q}} \neq 0$,  only if for all factors $\Exp X_{\mathbf{p}}X_{\mathbf{q}}$ indices $i_{p} = i_{q+1}$ and  $i_q = i_{p+1}$. So, if a chord connects the $p^{th}$ and $q^{th}$ vertices of the same color, say black, then  $\Exp X_{\mathbf{p}}X_{\mathbf{q}} =\Exp X_{i_p i_{p+1}}X_{i_q i_{q+1}} = \Exp X_{i_p i_{p+1}}X_{i_{p+1} i_{p}} = \rho$. If a chord connects a black vertex $p^{th}$ and a white vertex $q^{th}$, then  $\Exp X_{\mathbf{p}}X_{\mathbf{q}} =\Exp X_{i_p i_{p+1}}X_{i_{q+1} i_{q}} = \Exp X^2_{i_p i_{p+1}} = 1$. This implies that for any $d  \in \mathcal{D}^0_{2k}$
  $$\prod_{p \overset{d}\sim q} \Exp X_{\mathbf{p}}X_{\mathbf{q}} = \rho^{\# \textrm{chords connecting vertices of the same color}} = \omega(d) \, ,$$
  and so Lemma \ref{lem:comb} is proved.
\end{proof}

\begin{remark}
 While the proof of Lemma \ref{lem:comb} requires the assumption that the matrix entries are Gaussian random variables, the statement holds under much weaker assumptions. It seems to be possible to prove Lemma \ref{lem:comb} without the Gaussian assumption, using technique similar to \cite{goetze2014,alexeev2010a}. 
\end{remark}

 \begin{example}
 Let us consider the case $k=1$ and compute  $\frac{1}{N}\Exp \Tr W$. Chord diagrams, corresponding to all pair partitions in this case, are presented on Fig.\ref{fig:exk1}. 
 
 \begin{figure}[h]
 \centering
 \includegraphics[width=\textwidth]{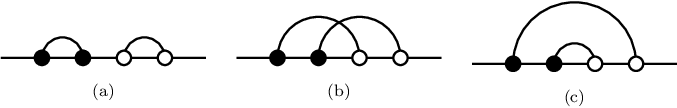}
  \caption{All chord diagrams which correspond to the expectation of the trace $\Exp X_{i_{0} i_{1}} X_{i_{1} i_{2}} X_{{i_3} i_{2}} X_{i_{0} i_{3}}$. The chord diagram (a) corresponds to the  pair partition $\Exp X_{i_{0} i_{1}} X_{i_{1} i_{2}} \Exp X_{{i_3} i_{2}} X_{i_{0} i_{3}}$, (b) corresponds to the pair partition  $\Exp X_{i_{0} i_{1}} X_{i_{3} i_{2}} \Exp X_{{i_1} i_{2}} X_{i_{0} i_{3}}$, and (c) corresponds to the pair partition $\Exp X_{i_{0} i_{1}} X_{i_{0} i_{3}} \Exp X_{{i_1} i_{2}} X_{i_{3} i_{2}}$.}\label{fig:exk1}
 \end{figure}
 
 Each product of the form $\Exp X_{i_j i_k} X_{i_l i_m}$ is not equal to $0$ only if $i_j = i_l, i_k = i_m$ or  $i_j = i_m, i_k = i_l$. For the case on Fig.\ref{fig:exk1}a it gives us
 \begin{align*}
  \Exp X_{i_{0} i_{1}} X_{i_{1} i_{2}} =\rho \ind{i_0=i_2} + \ind{i_0=i_1, i_1=i_2} \, , \\
  \Exp X_{{i_3} i_{2}} X_{i_{0} i_{3}} = \rho \ind{i_2=i_0} + \ind{i_3=i_0, i_2=i_3} \, , 
 \end{align*}
 and, finally, 
 \begin{align*}
  \Exp X_{i_{0} i_{1}} X_{i_{1} i_{2}} \Exp X_{{i_3} i_{2}} X_{i_{0} i_{3}} &=\\
 = \rho^2 \ind{i_2=i_0} + \rho\ind{i_0=i_2=i_3}&+\rho\ind{i_0=i_2=i_3}+\ind{i_0= i_1 = i_2=i_3} \, .
 \end{align*}
 
Analogously, the chord diagram on Fig.\ref{fig:exk1}b gives
 \begin{align*}
  \Exp X_{i_{0} i_{1}} X_{i_{3} i_{2}} \Exp X_{{i_1} i_{2}} X_{i_{0} i_{3}} &=\\
 = \rho^2 \ind{i_0=i_2, i_1 = i_3}& + (2\rho+1)\ind{i_0= i_1 = i_2=i_3} \, ,
 \end{align*}
 and the chord diagram on Fig.\ref{fig:exk1}c gives 
 \begin{align*}
  \Exp X_{i_{0} i_{1}} X_{i_{0} i_{3}} \Exp X_{{i_1} i_{2}} X_{i_{3} i_{2}} &=\\
 = \ind{i_1=i_3} + \rho\ind{i_0=i_1=i_3}&+\rho\ind{i_1=i_2=i_3}+\rho^2\ind{i_0= i_1 = i_2=i_3} \, .
 \end{align*}
 
 Using  Wick's formula, we obtain
 \small
 \begin{align*}
  &\frac{1}{N} \Exp Tr W = \frac{1}{N^3} \sum_{(4)} \Exp X_{i_{0} i_{1}} X_{i_{1} i_{2}} X_{{i_3} i_{2}} X_{i_{0} i_{3}} =  \\
   &\frac{1}{N^3} \sum_{(4)} (\Exp X_{i_{0} i_{1}} X_{i_{1} i_{2}} \Exp X_{{i_3} i_{2}} X_{i_{0} i_{3}} +\Exp X_{i_{0} i_{1}} X_{i_{3} i_{2}} \Exp X_{{i_1} i_{2}} X_{i_{0} i_{3}} +  \Exp X_{i_{0} i_{1}} X_{i_{0} i_{3}} \Exp X_{{i_1} i_{2}} X_{i_{3} i_{2}}) \, .   
 \end{align*}
 \normalsize
  Since the number of summands in the sum $\sum_{(4)}$ is equal to $N^f$, where $f$ is the number of independent indices $i_j$, or, in terms of chord diagrams, a number of boundary components,  we have
  
  $$\frac{1}{N} \Exp Tr W = \left(\rho^2+\frac{2\rho}{N}+\frac{1}{N^2} \right)+ \left(\frac{\rho^2}{N}+\frac{2\rho+1}{N^2} \right) +\left(1+\frac{2\rho}{N}+\frac{\rho^2}{N^2} \right) \; .$$
  
 In the large $N$ limit we get
 $$\lim_{N \to \infty}\frac{1}{N} \Exp Tr W = \rho^2+1 \; .$$
 

 \end{example}
 
 Let us now compute $\sum_{d \in \mathcal{D}^0_{2k}} \omega(d)$. To do this we introduce two sets $\mathcal{U}_k$ and $\mathcal{V}_k$ of planar chord diagrams.
 The set $\mathcal{U}_k$ contains all planar chord diagrams on $2k$ vertices, which are colored black and white according to the rule: 
 the $j^{th}$ vertex is black, if $j \pmod{4} \in \{1,2\}$ and white otherwise.  The set $\mathcal{V}_k$ contains all planar chord diagrams on $2k$ vertices, which are colored black and white according to the rule:  the $j^{th}$ vertex is black if $j \pmod{4} \in \{0,1\}$, and white otherwise. We use $\overline{\mathcal{U}}_k$ and  $\overline{\mathcal{V}}_k$ for the set of diagrams from   $\mathcal{U}_k$ and  $\mathcal{V}_k$ (respectively) with inverted colors.
 Note that the set $\mathcal{U}_{2k} = \mathcal{D}^0_{2k}$ for all $k$. We define partition functions $U_k(\rho)$ and $V_k(\rho)$:
 $$U_k(\rho) = \sum_{d \in \mathcal{U}_{k}} \omega(d) \, ,$$
  $$V_k(\rho) = \sum_{d \in \mathcal{V}_{k}} \omega(d) \, .$$
  We assume that $\mathcal{U}_0 = \mathcal{V}_0$ contains one empty diagram, so   $U_0(\rho) = V_0(\rho) =1$.
  
  \begin{lemma}
  The partition functions $U_k(\rho)$ and $V_k(\rho)$ satisfy the following recurrent relations:
  $$ U_{k+1}(\rho) = \sum_{i =0}^{\left\lfloor\frac{k-1}{2}\right\rfloor} U_{k-2i-1}(\rho) V_{2i+1}(\rho) + \rho \sum_{i=0}^{\left\lfloor\frac{k}{2}\right\rfloor} V_{2i}(\rho) U_{k-2i}(\rho) \, ,$$
  $$ V_{k+1}(\rho) = \sum_{i=0}^{\left\lfloor\frac{k}{2}\right\rfloor} U_{2i}(\rho) V_{k-2i}(\rho)+\rho \sum_{i=0}^{\left\lfloor\frac{k-1}{2}\right\rfloor} U_{2i+1}(\rho)V_{k-2i-1}(\rho) \, ,$$
  with initial conditions $U_0(\rho) = V_0(\rho) =1$.
  \label{lem:recrel}
  \end{lemma}
  
  \begin{proof}
   Consider a chord diagram $d \in \mathcal{U}_{k+1}$. Its first vertex is connected with some vertex with an even number, since the diagram is planar. If its first vertex is connected with the $(4i+4)^{th}$ vertex, then, after removing the first chord, the diagram $d$ splits into two diagrams:  $d_1 \in \mathcal{V}_{2i+1}$ and  $d_2 \in \mathcal{U}_{k-2i-1}$, respectively (see Fig.\ref{fig:usplit}a). The weight of such a diagram is $\omega(d) = \omega(d_1)\omega(d_2)$. For any pair of diagrams $d_1 \in \mathcal{V}_{2i+1}$ and  $d_2 \in \mathcal{U}_{k-2i-1}$ there exists a unique diagram $d$ which splits into $d_1$ and $d_2$ after the removing of its first chord. 
   
   If the  first vertex of $d$ is connected with the $(4i+2)^{nd}$ vertex, then, after removing the first chord, the diagram $d$ splits into two diagrams:  $d_1 \in \mathcal{V}_{2i}$ and  $d_2 \in \overline{\mathcal{U}}_{k-2i}$, respectively (see Fig.\ref{fig:usplit}b). The weight of such a diagram is  $\omega(d) = \rho\omega(d_1)\omega(d_2)$. Again, diagrams $d_1$ and $d_2$ determine the diagram $d$ in a unique way.
   
   It gives us
   \begin{align*}
   &U_{k+1}(\rho) = \sum_{d \in \mathcal{U}_{k+1}} \omega(d) = \\
   &\sum_{i =0}^{\left\lfloor\frac{k-1}{2}\right\rfloor} \sum_{d_1 \in \mathcal{V}_{2i+1}} \sum_{d_2 \in \mathcal{U}_{k-2i-1}} \omega(d_1)\omega(d_2) + \sum_{i=0}^{\left\lfloor\frac{k}{2}\right\rfloor} \sum_{d_1 \in \mathcal{V}_{2i}} \sum_{d_2 \in \overline{\mathcal{U}}_{k-2i}}\rho\omega(d_1)\omega(d_2) = \\
   &\sum_{i =0}^{\left\lfloor\frac{k-1}{2}\right\rfloor} \sum_{d_1 \in \mathcal{V}_{2i+1}}\omega(d_1)\, \sum_{d_2 \in \mathcal{U}_{k-2i-1}}\omega(d_2) + \rho\sum_{i=0}^{\left\lfloor\frac{k}{2}\right\rfloor} \sum_{d_1 \in \mathcal{V}_{2i}}\omega(d_1)\, \sum_{d_2 \in \overline{\mathcal{U}}_{k-2i}}\omega(d_2) =\\
   &\sum_{i =0}^{\left\lfloor\frac{k-1}{2}\right\rfloor} U_{k-2i-1}(\rho) V_{2i+1}(\rho) + \rho \sum_{i=0}^{\left\lfloor\frac{k}{2}\right\rfloor} V_{2i}(\rho) U_{k-2i}(\rho) \, .
   \end{align*}
   
   \begin{figure}[h]
 \centering
 \includegraphics[width=\textwidth]{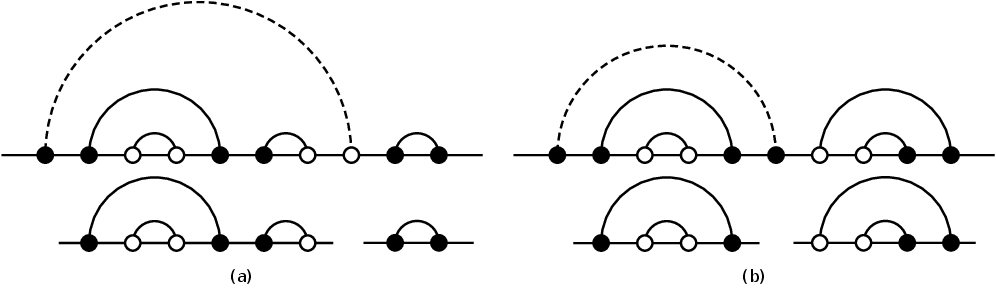}
  \caption{Two cases of splitting of a diagram $d \in \mathcal{U}_{5}$: (a) $d$ splits into $d_1 \in \mathcal{V}_{3}$  and $d_2 \in \mathcal{U}_{1}$; (b) $d$ splits into $d_1 \in \mathcal{V}_{2}$  and $d_2 \in \overline{\mathcal{U}}_{2}$.}\label{fig:usplit}
 \end{figure}
 Analogously one can compute $V_{k+1}(\rho)$.
   
  \end{proof}
  
  We have shown that for any natural $k$ the moment $M^{(N)}_k(\rho)$ of the distribution $F_N(x)$ has a limit $M_k(\rho)$. We have proved that $M_k(\rho) = U_{2k}(\rho)$, where polynomials $U_k(\rho)$ satisfy the recurrent relation \eqref{eq:rec}. Note that $U_{k}(\rho) \leq U_{k}(1)$ for any $\rho \in [-1,1]$ and $U_{k}(1)$ is equal to the $k^{th}$ Catalan number, because the equations \eqref{eq:rec}  with $\rho = 1$ turn into the classic relation for Catalan numbers. These reasons lead us to the inequalities $\sqrt[k]{M_k(\rho)} = \sqrt[k]{U_{2k}(\rho)} \leq \sqrt[k]{U_{2k}(1)} \leq 16$, so the moment problem is determined and the limiting distribution $F(x)$ has a finite support. Theorem \ref{th:mom} is proved.
  
  \begin{remark}
   We note that Theorem \ref{th:mom} can be generalized for higher powers of matrix $X$. For example, we propose that for  $W_3 =\frac{1}{N^3} X^3 X^{*3}$ a system of equations similar to \eqref{eq:rec} can be obtained, but such a system would contain $3$ equations and $3$ unknown functions. We also note that the case $W_1 = \frac{1}{N} X X^{*}$ is well studied and the spectral distribution in this case is known to be Marchenko--Pastur distribution \cite{naumov2012}.
  \end{remark}

 \section{Narayana Polynomials}
Narayana numbers, the coefficients of Narayana polynomials (of any type), are natural numbers, which occur in various counting problems, and their most general definition involves $h$-vectors of associahedra of a corresponding type (see  \cite{fomin2005}). Here we will define Narayana polynomials of types A and B using the notion of non-crossing partitions.

Let us define usual (type A) non-crossing partitions first.
\begin{definition}
 Let $S$  be a finite totally ordered set.
 We call $\pi = \{B_1, B_2, \dots, B_r \}$ a partition of the set $S$ if the $B_i \, (1 \leq i \leq r)$ are pairwise disjoint, non-empty subsets of $S$, such that
$B_1 \cup B_2 \cup \dots \cup B_r = S$. We call $B_i$ blocks of $\pi$. We use $p \underset{\pi}\sim q$ to denote that $p$ and $q$ belong to the same block of $\pi$.

A partition $\pi$ of the set $S$ is called crossing if there exist $p_1 < q_1 < p_2 < q_2$ in $S$ such that $p_1, p_2 \in B_i$, $q_1, q_2 \in B_j$ and $B_i\neq B_j$.

If $\pi$ is not crossing, then it is called {\it non-crossing}. The set of all non-crossing partitions of $S$ is denoted by $NC(S)$. If $S =  \{1,2,\dots,n\}$ (or, for brevity,  if $S = [n]$), then we use $NC(n)$ or $NC^A(n)$ for $NC(\{1,2,\dots,n\})$.
\end{definition}
We define type B non-crossing partitions in the same way as in \cite{reiner1997}.
\begin{definition}
 We say that a non-crossing partition $\pi$ of a set  
$$[\pm n] = \{-1, -2, \dots, -n, 1, 2, \dots, n \} \, ,$$
equipped with a total order:
$$-1 < -2 < \dots< -n< 1 < 2 < \dots < n \, ,$$
is a type B non-crossing partition (denoted by $\pi \in NC^B(n)$), if for any block $B \in \pi$, its negative $-B$ (obtained by negating all the elements of $B$) is also a block of $\pi$. Note that there is at most one block (called the zero block, if present) containing both $+i$ and $-i$ for some $i$.
\end{definition}
All $B$-type non-crossing partitions of $[\pm 2]$ are presented in Figure \ref{fig:btype}.
  \begin{figure}[h]
   \centering
   \includegraphics[width=\textwidth]{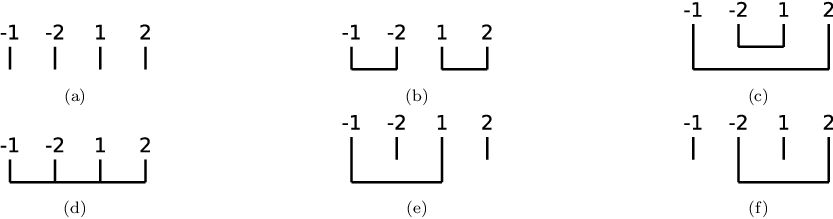}
   \caption{All $NC^B(2)$ partitions. Partitions (d), (e) and (f) have a zero block, and partitions (a), (b) and (c) do not.}
  \label{fig:btype}
  \end{figure}
\begin{definition}
 The combinatorial interpretation of Narayana polynomials is as follows:
\begin{align*}
 N^{A}_n(t) &= \sum_{\pi \in NC^A(n)} t^{\#\{\textrm{blocks in }\pi\}} \, , \\
  N^{B}_n(t) &= \sum_{\pi \in NC^B(n)} t^{\#\{\textrm{nonzero blocks in }\pi\}/2} \, .
\end{align*}
\end{definition}
Note that nonzero blocks appear in type B non-crossing partitions in pairs ($B$ and $-B$), so the exponent $\#\{\textrm{nonzero blocks in }\pi\}/2$ is always an integer number. By definition, for $n=0$ 
$$N^{A}_0(t)=N^{B}_0(t)=1 \, ,$$
and for $n>0$ the explicit formulas for Narayana polynomials \cite{reiner1997} are given by
\begin{align*}
 N^{A}_n(t) &= \sum_{k=1}^n \frac{1}{k}\binom{n-1}{k-1} \binom{n}{k-1} t^k \, , \\
  N^{B}_n(t)  &= \sum_{k=0}^n \binom{n}{k}^2 t^k \, .
\end{align*}

We will also use polynomials $Q_n(t)$ -- a derivative of Narayana polynomial of type A:
 $$Q_{n-1}(t) =  \sum_{k=1}^n \binom{n-1}{k-1} \binom{n}{k-1} t^{k-1} =  \left(N^{A}_n(t)\right)' \, .$$
 
We note here, that coefficients of $N^A_n(t), N^B_n(t)$ and $Q_n(t)$ are tabled in OEIS \cite{oeis} (sequences A001263, A008459, A132813, respectively).

The coefficients of $Q_{n-1}(t)$ have a natural interpretation. Let us consider $\pi \in NC^A(n)$ and mark one of its blocks. Then
\begin{equation}
 Q_{n-1}(t) = \sum_{\pi \in NC^A(n) \textrm{with one marked block}} t^{\#\{\textrm{unmarked blocks in }\pi\}} \, .
\end{equation}
Let us denote by $NC'(n)$ the set of all type $A$ non-crossing partitions of $[n]$  with a marked block.

Now we formulate several relations between $N^A_n(t), N^B_n(t)$ and $Q_n(t)$. Some of them are known, and some of them are probably new, so we provide proofs for them.
\begin{theorem} For all integer $n > 0$, the following relations hold:
 \begin{align}
  &N^B_n(t) = Q_{n-1}(t) + t^n Q_{n-1}\left(\frac{1}{t}\right) \, , \label{eq:NQQ} \\
  &(n+1)N^A_n(t) = t Q_{n-1}(t) + t^n Q_{n-1}\left(\frac{1}{t}\right) \, , \label{eq:NAQQ} \\
  &Q_{n-1}(t) = \sum_{k=1}^{n} N_{k-1}^{A}(t) N^{B}_{n-k}(t) \, , \label{eq:QNN}\\
  &Q_{n}(t) =  (n+1)N_{n}^A(t) + \sum_{k=1}^{n} N_{k-1}^{A}(t) Q_{n-k}(t)\, , \label{eq:QNNQ}\\
  &N^A_n(t) = t N^A_{n-1}(t) + \sum_{k=1}^{n-1}N^A_{k-1}(t) N^A_{n-k}(t) \, ,\label{eq:NAr} \\
  &N^B_n(t) =  t N^B_{n-1}(t)+\sum_{k=1}^n N^A_{k-1}(t) N^B_{n-k}(t) + \sum_{k=1}^{n-1} N^B_{k-1}(t) N^A_{n-k}(t) \, .  \label{eq:NBr}
 \end{align}
\label{th:comb}
\end{theorem}

\begin{proof}
Let us define the map 
$$\abs: NC^B(n) \to NC^A(n)$$ 
in the following way. For any non-crossing partition  $\pi \in NC^B(n)$ we define a non-crossing partition $\sigma = \abs(\pi)  \in  NC^A(n)$, in which $i\underset{\sigma}\sim j$ if and only if $i \underset{\pi}\sim j$ or $i \underset{\pi}\sim -j$. The map $\abs()$ is an $(n+1)-to-1$ map \cite{biane2003} and it respects the block statistics, that is: if $\pi \in NC^B(n)$ contains $2k$ nonzero blocks and $z$ zero blocks ($z$ can be equal to  $1$ or $0$), then $\sigma = \abs(\pi)$ contains $k+z$ blocks.

Let us also define the Kreweras complementation map $K$ in the following way. Let $\pi$ be a non-crossing partition of an arbitrary  finite totally ordered set. We define a non-crossing partition $\lambda = K(\pi)$, in which $i\underset{\lambda}\sim j$ if and only if there are no pairs $k\underset{\pi}\sim  l$, such that $k \leq i < l \leq j$ or $i < k \leq j < l$. In particular, if $\pi \in NC(n)$, consider a totally ordered set $S(n,n')$ 
 $$1<1'<2<2'< \dots< n< n'$$ and $\pi$ - a non-crossing partition on its subset $[n]$. Then $\lambda = K(\pi)$ is the largest (with respect to the refinement order) non-crossing partition on $\{1',2',\dots,n'\}$, such that $\pi \cup \lambda$ is a non-crossing partition on $S(n,n')$  (See Figure \ref{fig:kreweras} for an example).
  \begin{figure}[ht]
   \centering
   \includegraphics[width=\textwidth]{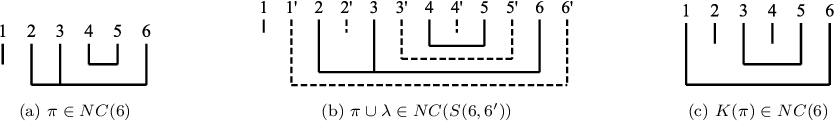}
   \caption{An example of the Kreweras complementation map.}
   \label{fig:kreweras}
  \end{figure}

 Denote by $NC_0^B(n)$ the set of partitions from $NC^B(n)$, containing a zero block, and by $Z(\pi)$ the zero block of $\pi \in NC_0^B(n)$. For any $\pi \in NC_0^B(n)$ we consider $\abs(\pi)$ and mark an image of its zero block.  This procedure is a bijection between $NC_0^B(n)$ and $NC'(n)$. The bijection is statistics-preserving, i.e., half the number of nonzero blocks in $\pi \in NC_0^B(n)$ is equal to the number of unmarked blocks in its image $\sigma' \in NC'(n)$. It implies:
 \begin{equation}
   Q_{n-1}(t) =  \sum_{\substack{\pi \in NC^B(n):\\ \pi\textrm{ has a zero block}}} t^{\#\{\textrm{nonzero blocks in }\pi\}/2} \, .
   \label{eq:Qnz}
 \end{equation}

 An example with $n=3$ is given in Figure \ref{fig:bt_qt}.
   \begin{figure}[ht]
   \centering
   \includegraphics[width=\textwidth]{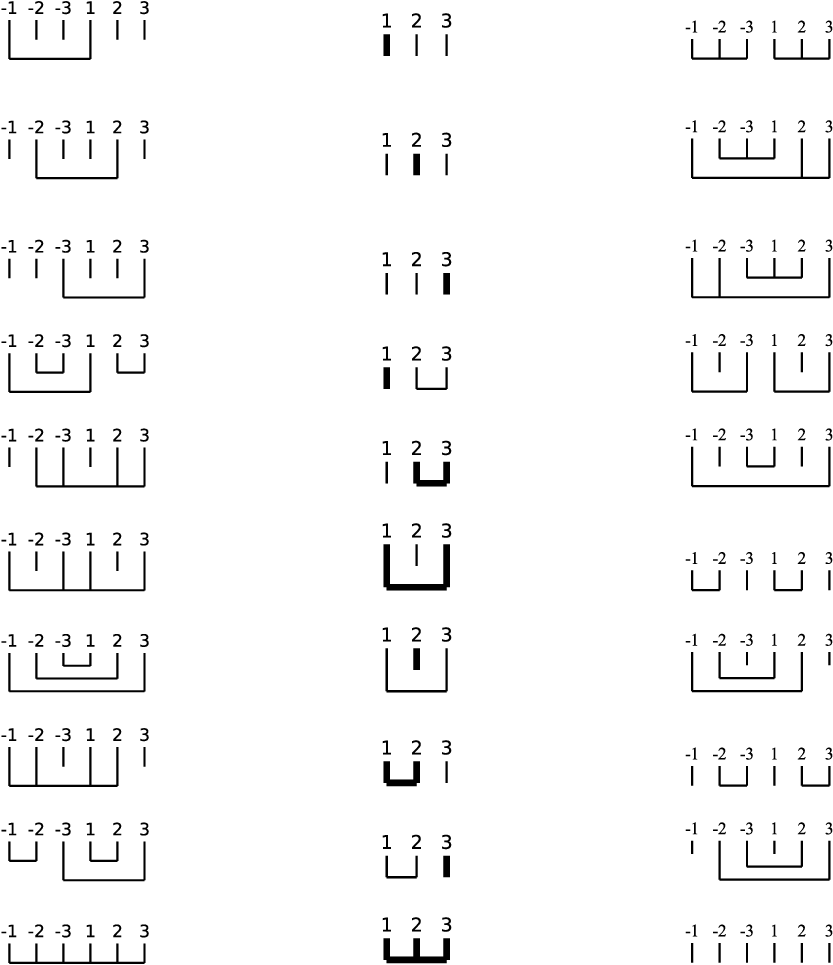}
   \caption{Left column: partitions $NC_0^B(3)$; Center column: corresponding elements of $NC'(3)$; Right column: image of left column under map $K$.}
   \label{fig:bt_qt}
  \end{figure}
  
 Furthermore, the Kreweras map is a bijection $K$ between B-type non-crossing partitions, containing a zero block, and B-type non-crossing partitions, containing no zero blocks \cite{reiner1997}. The map $K$ has the following property: if $\pi \in NC_0^B(n)$ contains $2k+1$ blocks ($2k$ nonzero blocks and $1$ zero block), then $K(\pi)$ contains $2n-2k$ nonzero blocks. This implies
 \begin{equation}
    t^n Q_{n-1}\left(\frac{1}{t}\right) =  \sum_{\substack{\pi \in NC^B(n):\\ \pi\textrm{ has no zero block}}} t^{\#\{\textrm{nonzero blocks in }\pi\}/2} \, ,
    \label{eq:Qwoz}
 \end{equation}

 and proves \eqref{eq:NQQ}.
 
 To prove \eqref{eq:NAQQ} we consider the image of $NC^B(n)$ under the map $\abs()$. Note that if $\pi \in NC_0^B(n)$ contains $2k+1$ blocks ($2k$ nonzero blocks and $1$ zero block), then $\abs(\pi)$ contains $k+1$ blocks, and if $\pi \in NC^B(n)$ contains no zero block and $2k$ nonzero blocks, then $abs(\pi)$ contains $k$ blocks. Since $abs()$ is a $(n+1)-to-1$ map, we get
 $$(n+1)N^A_n(t) = t Q_{n-1}(t) + t^n Q_{n-1}\left(\frac{1}{t}\right) \, .$$

  To prove \eqref{eq:QNN} we will use an interpretation of $Q_n(t)$ in terms of $B$-type non-crossing partitions with a zero-block \eqref{eq:Qnz}. For each $\pi \in NC_0^B(n)$ we define $-l$ -- the smallest element of the zero block, and 
  $$m = \max \{i>0: -i  \not\underset{\pi}\sim i \textrm{ and } -i  \underset{\pi}\sim j \textrm{ for some } j > i\}$$ (if the set is empty, we assume $m = 0$).  We note that the set $[m+1,l-1]$ is closed with respect to $\pi$, i.e. any block of $\pi$ either is a subset of $[m+1,l-1]$ or does not intersect it. We decompose $\pi$ into a pair $\pi_A \in NC^A(k-1)$ and $\pi_B \in NC^B(n-k)$, where $k = l-m$. The structure of $\pi_A$ is inherited from $\pi$ on the set $[m+1,l-1]$, i.e. for any $1\leq i,j \leq l-m-1$ we have $i \underset{\pi_A}\sim j$ if and only if $i+m \underset{\pi}\sim j+m$. The structure of $\pi_B$ is inherited from $\pi$ on the set $[\pm n] \setminus [-(m+1),-l]\setminus [m+1,l]$.
  The described decomposition of $\pi$ into $\pi_A$ and $\pi_B$ is the bijection between $NC_0^B(n)$ and $\cup_{k=1}^n NC^A(k-1)\times NC^B(n-k)$. The inverse map is the following. To reconstruct $\pi \in  NC_0^B(n)$ from given $\pi_A \in NC^A(k-1)$ and $\pi_B \in NC^B(n-k)$ one needs to add an element $l$ (and its negative $-l$) to the zero block of $\pi_B$ (or create the zero block, if needed) in the leftmost possible position and then add $\pi_A$ and its negative copy in front of $l$ and $-l$, respectively. 
  
   This proves \eqref{eq:QNN}. An example of the decomposition is given in Figure \ref{fig:QNN}.


  \begin{figure}[ht]
   \centering
   \includegraphics[width=\textwidth]{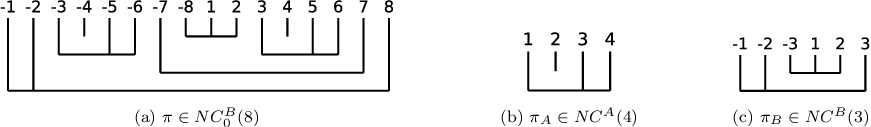}
  \caption{Decomposition of $\pi$ into $\pi_A$ and $\pi_B$.}
  \label{fig:QNN}
  \end{figure}

 
 The proof of \eqref{eq:QNNQ} is similar. First, consider the case when the zero block of $\pi \in NC_0^B(n+1)$ contains only 2 elements. We note that there is a natural statistics-preserving $(n+1)-to-1$ map between $\{\pi \in NC_0^B(n+1): |Z(\pi)| =2\}$ and $NC^A(n)$ -- one can delete the zero block and then take $\abs()$. Indeed, the preimage $\pi \in NC_0^B(n+1)$ of $\pi_A \in NC^A(n)$ has a two-element zero block $(-l,l)$ with some $l \in \{1,2,\dots,n+1\}$, and the rest of the structure is uniquely defined by the structure of $\pi_A$ and the position of the zero block. So,
 $$\sum_{\substack{\pi \in NC_0^B(n+1):\\ |Z(\pi)| =2}} t^{\#\{\textrm{nonzero blocks in }\pi\}/2} = (n+1)N_{n}^A(t) \, .$$
 
 Then, if the zero block of $\pi \in NC_0^B(n+1)$ contains more than $2$ elements, one can decompose such a partition into $\pi_A \in NC^A(k-1)$ and $\pi_B \in NC_0^B(n-k)$. Namely, denote by $-l_1$ and $-l_2$ the first and the second elements of the zero block, respectively, and assume $k = l_2-l_1$. The structure of $\pi_A$ is inherited from $\pi$ on the set $[l_1+1,l_2-1]$ (this set can be empty). The structure of $\pi_B$ is inherited from $\pi$ on the set $[\pm n] \setminus [-l_1,-(l_2-1)]\setminus [l_1,l_2-1]$. It concludes the proof.
 
 For the proof of \eqref{eq:NAr} and \eqref{eq:NBr} see  \cite[Proposition 5]{reiner1997}.
\end{proof}

\begin{cor}
 Let $P^A_n(t), P^B_n(t)$ and $R_n(t)$ be families of polynomials, such that
  \begin{align}
  &P^B_n(t) = t^n R_{n-1}(1/t) + \sum_{k=1}^{n} P_{k-1}^{A}(t) P^{B}_{n-k}(t) \, , \label{eq:exact1} \\ 
  & t^n R_{n}(1/t) = P^B_n(t) + \sum_{k=1}^{n} t^{n-k} P_{k-1}^{A}(t)R_{n-k}(1/t) \, , \label{eq:exact2} \\
  &P^A_n(t) = \sum_{k=1}^{n} t^k P^A_{k-1}(1/t) P^A_{n-k}(t) \, , \label{eq:exact3} \\ 
 \end{align}
 and $P^A_0(t) = 1$, $P^B_0(t) = 1$, $R_0(t) =1$. Then for all $n \geq 0$ 
 $$P^A_n(t) = N^A_n(t), P^B_n(t) =  N^B_n(t) \textrm{ and } R_n(t) = Q_n(t).$$
 \label{cor:nar}
\end{cor}

\begin{proof}
First we show that the polynomials  $N^A_n(t), N^B_n(t)$ and $Q_n(t)$ satisfy equations \eqref{eq:exact1}, \eqref{eq:exact2}, \eqref{eq:exact3}.

Combining \eqref{eq:NQQ} and \eqref{eq:QNN} we obtain
$$N^B_n(t) =t^n Q_{n-1}(1/t)+ \sum_{k=1}^{n} N_{k-1}^{A}(t) N^{B}_{n-k}(t) \, .$$

Combining \eqref{eq:QNNQ} and \eqref{eq:NAQQ} we obtain
$$Q_{n}(t) =  t Q_{n-1}(t) + t^n Q_{n-1}(1/t) + \sum_{k=1}^{n} N_{k-1}^{A}(t) Q_{n-k}(t) \, .$$
Substituting $t$ and $1/t$ and multiplying by $t^n$ we get
$$t^n Q_{n}(1/t) =  t^{n-1} Q_{n-1}(1/t) + Q_{n-1}(t) + \sum_{k=1}^{n} t^n N_{k-1}^{A}(1/t) Q_{n-k}(1/t) \, .$$
Note that $N_k^A(t)=t^{k+1}N^A_k(1/t) \textrm{ for all } k > 0$.  Using this we have
$$\sum_{k=1}^{n} t^n N_{k-1}^{A}(1/t) Q_{n-k}(1/t) = t^n Q_{n-1}(1/t)+\sum_{k=2}^{n} t^{n-k} N_{k-1}^{A}(t) Q_{n-k}(1/t) \, ,$$
and using \eqref{eq:NQQ} we finally get
$$t^n Q_{n}(1/t) =  N^B_n(t) + \sum_{k=1}^{n} t^{n-k} N_{k-1}^{A}(t) Q_{n-k}(1/t) \, .$$
The facts $N_k^A(t)=t^{k+1}N^A_k(1/t) \textrm{ for all } k > 0$ and \eqref{eq:NAr} imply 
$$N^A_n(t) = \sum_{k=1}^{n} t^k N^A_{k-1}(1/t) N^A_{n-k}(t) \, .$$

Since the equations \eqref{eq:exact1}, \eqref{eq:exact2}, \eqref{eq:exact3} together with initial conditions determine the polynomial families $P^A_n(t), P^B_n(t)$ and $R_n(t)$ uniquely (by induction),  polynomials $N^A_n(t), N^B_n(t)$ and $Q_n(t)$ represent the only solution of the system.
\end{proof}

\section{Free Cumulants}\label{sec:freecumul}
In this section we prove Theorem \ref{th:cumul}. 
At first we recall some notations and assertions concerning the relations between moments and free cumulants. For more details see \cite{speicher2006}.

Let $\mu$ be some probability distribution with a finite support, $M_n$ be its moments and $c_n$ be its free cumulants.
Then 
\begin{align}
 M_n = \sum_{\pi \in NC(n)} \prod_{B \in \pi} c_{|B|} \, .
 \label{eq:momcumul}
\end{align}

We will apply \eqref{eq:momcumul}  to the distribution $G(x)$. Its odd moments are equal to zero, and its even moments
$$ \widetilde{M}_{2k}(\rho) = \sum_{d \in \mathcal{U}_{2k}} \omega(d) \, . $$
Since odd moments of $G(x)$ are equal to zero, odd free cumulants of $G(x)$ are equal to zero as well. 
We will assign planar chord diagrams to non-crossing partitions and we will prove Theorem \ref{th:cumul} using double counting for \eqref{eq:momcumul}.
\begin{definition}
We say that a contiguous interval $S$ of a chord diagram $d$ is {\it closed}, if for any vertex $i \in S$ there exists a vertex $j \in S$ such that $i \overset{d}{\sim} j$.
 We say that a diagram $d \in \mathcal{U}_{k}$ is {\it decomposable} (see Fig.\ref{fig:cd_decom}a), if for some $l > 0$ there exists a closed contiguous interval consisting of $4l$ vertices, such that the vertex following it is the end of some chord (and we call such an interval {\it bad}). In the other case, we say that the diagram is {\it atomic}. We denote the set of atomic diagrams on $4k$ vertices by $\mathcal{B}_{2k}$ and  the set of atomic diagrams on $4k+2$ vertices by $\mathcal{\widetilde{B}}_{2k+1}$ .
\end{definition}

\begin{lemma}
The $n^{th}$ type B Narayana polynomial of $\rho^2$ is equal to the partition function of $\mathcal{B}_{2n}$:
$$\sum_{k=0}^n \binom{n}{k}^2 \rho^{2k} = \sum_{d \in \mathcal{B}_{2n}} \omega(d) \, .$$
\label{lem:main}
\end{lemma}
\begin{proof}
 Let us define two auxiliary sets of chord diagrams $\mathcal{A}_{2k+1}$ and  $\widetilde{\mathcal{A}}_{2k+1}$. The set $\mathcal{A}_{2k+1}$ contains all atomic diagrams from $\mathcal{U}_{2k+1}$, such that for any $l >0$ an interval $(4k-4l+1, 4k+2)$ is not closed. The set $\widetilde{\mathcal{A}}_{2k+1}$ is the subset of $\mathcal{V}_{2k+1}$ with the same restrictions. Denote by $B_{2k}(\rho)$, $\widetilde{B}_{2k+1}(\rho)$ $A_{2k+1}(\rho)$ and $\tilde{A}_{2k+1}(\rho)$ the partition functions of $\mathcal{B}_{2k}$, $\mathcal{\widetilde{B}}_{2k+1}$, $\mathcal{A}_{2k+1}$ and  $\widetilde{\mathcal{A}}_{2k+1}$, respectively.
 Consider the first chord of a diagram $d \in \mathcal{B}_{2k}$. It can connect the $1^{st}$ vertex with either the second vertex or the $4i^{th}$ vertex, where $0< i \leq k$. In the first case, the rest of the diagram belongs to $\mathcal{\widetilde{B}}_{2k-1}$, in the second case, the chord $(1,4i)$ splits the diagram into $d_1 \in  \widetilde{\mathcal{A}}_{2i-1}$ and $d_2 \in  {\mathcal{B}}_{2k-2i}$.   Using the same standard arguments for  $\mathcal{\widetilde{B}}_{2k+1}$, $\mathcal{A}_{2k+1}$ and  $\widetilde{\mathcal{A}}_{2k+1}$, we obtain:
 \begin{align}
  &B_{2k}(\rho) = \rho \widetilde{B}_{2k-1}(\rho) + \sum_{i =1}^{k} \tilde{A}_{2i-1}(\rho) B_{2k-2i}(\rho) \, ,\nonumber \\
  &\widetilde{B}_{2k+1}(\rho) = \rho B_{2k}(\rho) + \sum_{i =1}^{k} \tilde{A}_{2i-1}(\rho) \widetilde{B}_{2k+1-2i}(\rho) \, ,\nonumber \\
  &A_{2k+1}(\rho) = \sum_{i=1}^{k} \tilde{A}_{2i-1}(\rho) A_{2k-2i+1}(\rho)  \, , \nonumber \\
  &\tilde{A}_{2k+1}(\rho) = \rho\sum_{i=1}^{k} {A}_{2i-1}(\rho) \tilde{A}_{2k-2i+1}(\rho)  \, .
  \label{eq:narayana}
 \end{align}
 
 These  equations \eqref{eq:narayana} are equivalent to the equations in Corollary \ref{cor:nar}, and so polynomials $A_k(\rho)$, $B_k(\rho)$ coincide with the corresponding Narayana polynomials:
 \begin{align*}
  A_{2n+1}(\rho) &= \rho^{2n+1}N^{A}_n\left(\frac{1}{\rho^2}\right) \, , \\
  \tilde{A}_{2n+1}(\rho) &= N^{A}_n(\rho^2) \, , \\
  \widetilde{B}_{2n+1}(\rho) &=  \rho^{2n+1} Q_n\left(\frac{1}{\rho^2}\right) \, , \\
  B_{2n}(\rho) &= N^{B}_n(\rho^2) \, . \\
 \end{align*}
\end{proof}

  \begin{figure}[ht]
   \centering
   \includegraphics[width=\textwidth]{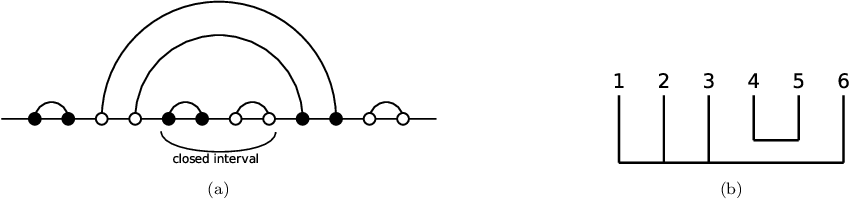}
  \caption{(a) Decomposable chord diagram. The interval $S = \{5,6,7,8\}$ is the only bad interval. (b) The resulting non-crossing partition.}
  \label{fig:cd_decom}
  \end{figure}


To finish the proof of Theorem \ref{th:cumul} we assign to any chord diagram $d \in  \mathcal{U}_{2n}$ a non-crossing partition $\pi \in NC(2n)$. If $d$ is atomic, then it corresponds to a non-crossing partition with one block of size $2n$. If $d$ is decomposable, then consider its leftmost (maximal by inclusion) bad interval $S$. Let $S$ have $4l$ vertices and let the $j^{th}$ chord end in the vertex directly following the interval $S$. Consider the interval $S$ and its complement $S^c$ as separate diagrams. The non-crossing partition, corresponding to $d$, splits into the non-crossing partition 
$$\lambda \in NC(\{1,2,\dots,j,j+2l+1,\dots,2n\}) \, ,$$
corresponding to $S^c$, and the non-crossing partition 
$$\sigma \in NC(\{j+1,j+2,\dots, j+2l\})\, ,$$ 
corresponding to $S$. The separating process results in a collection of atomic diagrams (after a finite number of steps) (See Fig.\ref{fig:cd_decom} for an example. In this example $l=1$ and $j=3$). 
From this decomposition rule, Lemma \ref{lem:main}, and Equation \eqref{eq:momcumul} follows, that corresponding free cumulants of the distribution $G(x)$ are Narayana polynomials of type $B$. Since $G(x)$ has a finite support, it is uniquely determined by its free cumulant.

According to Reiner \cite{reiner1997}, the generating function $f(x,t)$ of Narayana polynomials of type $B$ is given by 
$$f(x,t) = \sum_{n=0}^\infty N^B_n(t) x^n = \frac{1}{\sqrt{(1-(t-1) x)^2-4 x}} \, . $$
By the definition of free cumulants, we have
$$\mathcal{R}_G(z) = \frac{1}{z}\left(f(z^2,\rho^2)-1\right) = \frac{1}{z}\left(\frac{1}{\sqrt{\left((\rho^2-1) z^2-1\right)^2-4 z^2}}-1\right) \, .$$
Theorem \ref{th:cumul} is proved.

\section{Spectral Density} \label{sec:density}
The Cauchy transform $s_G(z)$ of the distribution $G$ can be found as a solution of 
$$\mathcal{R}_G (s_G(z))+\frac{1}{s_G(z)}=z \, ,$$
or, explicitly,
\begin{align}
 \frac{1}{s_G(z) \sqrt{\left(\rho^2-1\right)^2 s^4_G(z) -2(\rho^2+1)s^2_G(z)+1}} = z \, .
 \label{eq:s_g}
\end{align}

For the Cauchy transform $s_\mu(z)$ of any probability measure $\mu$ the following properties hold:
\begin{align}
\Im(s_\mu(z)) < 0, \textrm{ if } \Im(z) >0 \, ,  \label{eq:cond1}  \\
\lim_{y \to \infty} \sqrt{-1}y s_\mu(x+\sqrt{-1}y) = 1 \, .  \label{eq:cond2}
\end{align}
Equation \eqref{eq:s_g} has a unique solution which satisfies \eqref{eq:cond1}, \eqref{eq:cond2}. The explicit formulas of $s_G(z)$ for the cases $\rho = 0$ (Fuss--Catalan distribution) and $\rho = 1$ (squared Marchenko--Pastur distribution) are well known (see, e.g., \cite{zyczkowski2011}). In the general case, the solution can be obtained, but does not have a compact form. It still allows us to obtain a density of the distribution $G$ and the distribution $F$ at least numerically (see Fig.~\ref{fig:density}).
Indeed, the density $d_G(x)$ of the distribution $G$ can be obtained as
$$d_G(x) = \lim_{y \to 0} \frac{\Im(s_\mu(x+\sqrt{-1}y))}{\pi} \, ,$$
and the density $d_F(x)$ of the distribution $F$ can be obtained as
$$d_F(x) = \frac{d_G(\sqrt{x})}{\sqrt{x}} \, .$$
  \begin{figure}[ht]
  \centering
    \includegraphics[width=0.99\textwidth]{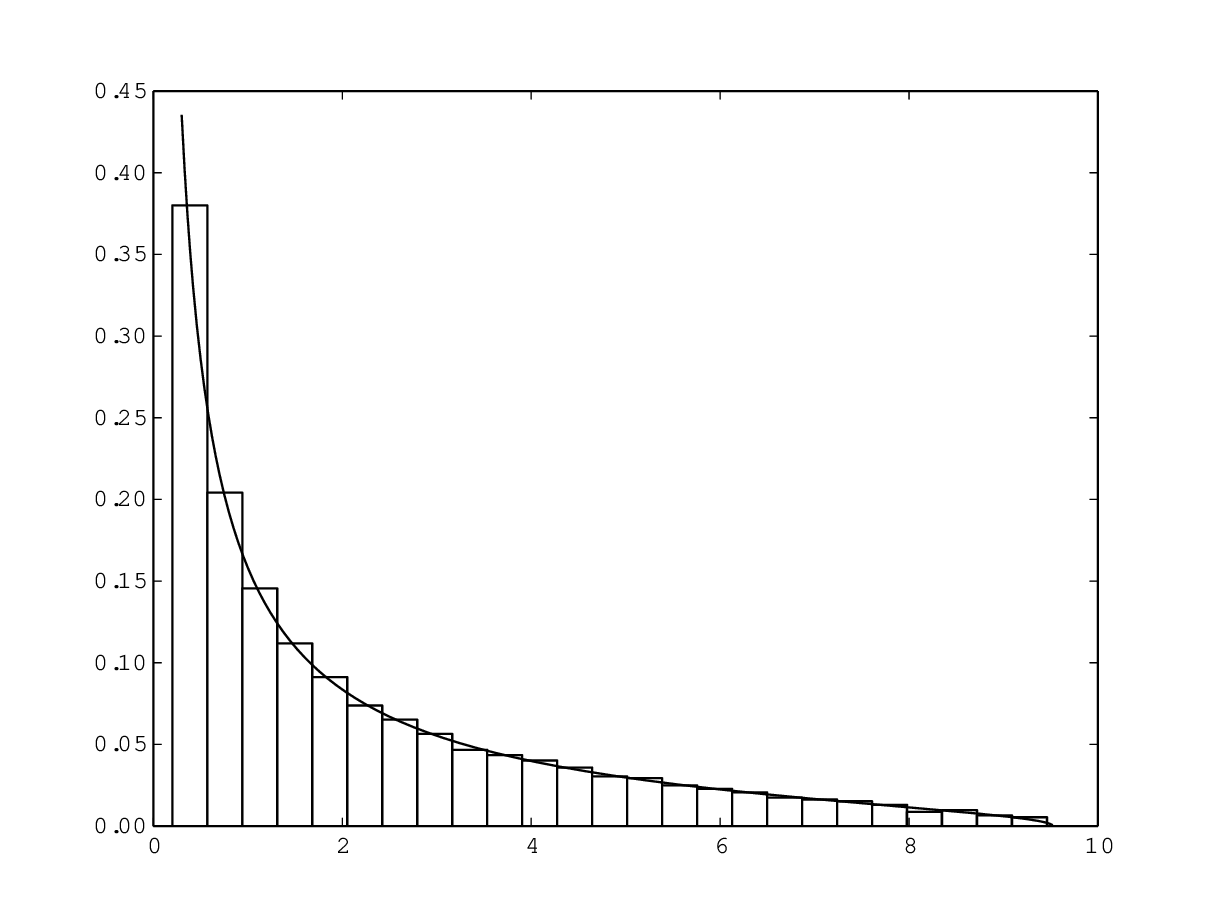}
  \caption{The density $d_F(x)$ and a histogram of the eigenvalues of $X^2 X^{*2}$, where $X$ is a $2500\times2500$ elliptic random matrix ($\rho = \frac{1}{2}$).}
  \label{fig:density}
 \end{figure}

\section{Conclusions}
In the current paper, we consider an asymptotic distribution of the singular values of the matrix $X^2$, where $X$ is an elliptic random matrix. While singular value distributions of powers and products of random matrices with independent entries have been studied extensively (see \cite{burda2010,alexeev2010c,penson2011,mlotkowski2012,akemann2013,lenczewski2014,forrester2014,forrester2015,speicher2015}, just to name a few), the recent results for elliptic random matrices (\cite{orourke2014a,orourke2014b,goetze2014}) concern mostly eigenvalue statistics. The obtained asymptotic distribution $F$ is a new generalization of Fuss--Catalan distribution (see \cite{mlotkowski2010,arizmendi2012,mlotkowski2014,lenczewski2013}). We find moments of this distribution, and prove that its free cumulants are Narayana Polynomials of type B. This means that for squares of elliptic random matrices, type B Catalan structures (say, non-crossing partitions of type B) play the same role as type A Catalan structures play for  products of two rectangular random matrices with independent entries.

While the main part of the paper is devoted to the combinatorial properties of the distribution $F$, our technique also allows to derive the density of $F$ (see Section \ref{sec:density}). 

\section*{Acknowledgements}
 The authors are grateful for fruitful discussions with Pavel Galashin and Mikhail Basok.
The work of NA is supported by the grant of Russian Scientific Foundation 14-11-00581. The work of AT is supported by SFB 701 “Spectral Structures and Topological Methods in Mathematics” University of Bielefeld, by RFBR grant RFBR 14-01-00500 and by Program of Fundamental Research Ural Division of RAS, Project 12-P-1-1013.

\bibliographystyle{spmpsci}
\bibliography{randmatr.bib}

\end{document}